\documentclass[12pt, leqno]{amsart}

\usepackage{float}
\usepackage{graphics}
\usepackage{amssymb,amsmath,amsthm,amscd}
\usepackage{mathrsfs}
\usepackage{enumerate}
\usepackage{color}
\usepackage[vcentermath,enableskew,noautoscale]{youngtab}
\usepackage[all]{xy}
\usepackage{verbatim}

\numberwithin{equation}{section}

\newcommand{\C}{{\mathbb C}}

\newcommand{\Z}{{\mathbb Z}}
\newcommand{\B}{{\mathbf{B}}}

\newcommand{\V}{{\mathbf V}}

\newcommand{\gl}{{\mathfrak{gl}}}
\newcommand{\q}{{\mathfrak{q}}}

\newcommand{\R}{{\mathbb R}}

\newcommand{\seteq}{\mathbin{:=}}

\theoremstyle{plain}
\newtheorem{lemma}{Lemma}[section]
\newtheorem{prop}[lemma]{Proposition}
\newtheorem{theorem}[lemma]{Theorem}
\newcommand{\Prop}{\begin{prop}}
\newcommand{\enprop}{\end{prop}}
\newcommand{\Lemma}{\begin{lemma}}
\newcommand{\enlemma}{\end{lemma}}
\newcommand{\Th}{\begin{theorem}}
\newcommand{\enth}{\end{theorem}}
\newtheorem{corollary}[lemma]{Corollary}
\newcommand{\Cor}{\begin{corollary}}
\newcommand{\encor}{\end{corollary}}
\newtheorem{definition}[lemma]{Definition}

\newcommand{\Def}{\begin{definition}}
\newcommand{\edf}{\end{definition}}

\theoremstyle{definition}
\newtheorem{remark}[lemma]{Remark}
\newtheorem{example}[lemma]{Example}

\newcommand{\g}{{\mathfrak{g}}}

\newcommand{\Hom}{\operatorname{Hom}}

\newcommand{\isoto}[1][]{\mathop{\xrightarrow[#1]%
{{\raisebox{-.6ex}[0ex][-.6ex]{$\mspace{2mu}\sim\mspace{2mu}$}}}}}
\newcommand{\tensor}{\otimes}

\newcommand{\eq}{\begin{eqnarray}}
\newcommand{\eneq}{\end{eqnarray}}

\newcommand{\eqn}{\begin{eqnarray*}}
\newcommand{\eneqn}{\end{eqnarray*}}

\newcommand{\on}{\operatorname}

\newcommand{\QED}{\end{proof}}
\newcommand{\Proof}{\begin{proof}}

\newcommand{\soplus}{\mathop{\mbox{\normalsize$\bigoplus$}}\limits}

\newcommand{\cl}{\colon}

\newcommand{\ba}{\begin{array}}
\newcommand{\ea}{\end{array}}

\newcommand{\bi}{\begin{enumerate}[{\rm(i)}]}

\newcommand{\set}[2]{\left\{#1 \mathbin{;} #2 \right\}}

\newcommand{\eqsub}{\begin{subequations}\begin{eqnarray}}
\newcommand{\eneqsub}{\end{eqnarray}\end{subequations}}

\newcommand{\ol}{\overline}

\newcommand{\nc}{\newcommand}
\nc{\la}{\lambda}
\nc{\lam}{\lambda}
\nc{\U}[1][\g]{U_q(#1)}
\nc{\te}{\tilde{e}}
\nc{\tei}{\tilde{e}_i}
\nc{\tf}{\tilde{f}}
\nc{\tfi}{\tilde{f}_i}
\nc{\tU}{\widetilde U_q(\g)}
\nc{\tE}{\tilde{E}}
\nc{\tF}{\tilde{F}}

\nc{\tk}{\tilde{k}}
\nc{\tkone}{\tk_{\ol{1}}}
\nc{\teone}{\tilde{e}_{\ol{1}}}
\nc{\tfone}{\tilde{f}_{\ol{1}}}

\nc{\teibar}{\tilde{e}_{\ol{i}}} \nc{\tfibar}{\tilde{f}_{\ol{i}}}
\nc{\tki}{{\tk}_{\ol {i}}}

\nc{\BZ}{{\mathbb{Z}}}
\nc{\al}{\alpha}
\nc{\qs}{{q}}
\nc{\lan}{\langle}
\nc{\ran}{\rangle}
\nc{\re}{{\mathrm{re}}}
\nc{\wt}{\operatorname{wt}}
\nc{\ch}{\operatorname{ch}}
\nc{\Uf}[1][\g]{U^-_q(#1)}
\nc{\Ue}{U^+_q(\g)}
\nc{\eps}{\varepsilon}
\nc{\vphi}{\varphi}
\nc{\sphi}{\varphi^*}
\nc{\seps}{\varepsilon^*}

\nc{\nn}{\nonumber}
\def\max{{\mathop{\mathrm{max}}}}
\nc{\vp}{\varpi}
\nc{\cls}{{\operatorname{cl}}}
\nc{\Wt}{{\operatorname{Wt}}}
\nc{\Us}{U'_q(\g)}
\nc{\La}{\Lambda}
\nc{\ro}{{\rm(}}
\nc{\rf}{{\rm)}}
\nc{\norm}{{\mathrm{norm}}}
\nc{\qbox}{\quad\mbox}
\nc{\braid}{{\mathfrak{B}}}
\nc{\Ad}{\operatorname{Ad}}
\nc{\Aut}{\operatorname{Aut}}
\nc{\dt}[1]{\tilde{\tilde #1}}
\nc{\Sn}{S^{{\mathrm{norm}}}}
\nc{\aff}{{\mathrm{aff}}}
\nc{\rk}{{\mathrm{rk}}}
\nc{\tQ}{\widetilde{Q}}
\nc{\tP}{\widetilde{P}}
\nc{\tW}{\widetilde{W}}
\nc{\Dyn}{\mathrm{Dyn}}
\nc{\tD}{\widetilde{\Delta}}
\nc{\height}{{\operatorname{ht}}}
\nc{\bl}{\bigl}
\nc{\br}{\bigr}
\nc{\Hecke}{\mathrm{H}}
\nc{\HA}{\Hecke^{\mathrm{A}}}
\nc{\HB}{\Hecke^{\mathrm{B}}}
\nc{\K}{\mathrm{K}}
\newcommand{\scbul}{{\,\raise1pt\hbox{$\scriptscriptstyle\bullet$}\,}}
\nc{\vac}{{\vphi}}
\nc{\Bt}{\B_\theta(\g)}
\nc{\be}{\begin{enumerate}}
\nc{\ee}{\end{enumerate}}
\nc{\low}{{\mathrm{low}}}
\nc{\upper}{{\mathrm{up}}}
\nc{\Zodd}{\Z_{\mathrm{odd}}}
\nc{\Ft}[1][n]{\mathbb{P}\mathrm{ol}_{#1}}
\nc{\Ftf}[1][n]{\widetilde{\mathbb{P}\mathrm{ol}}_{#1}}
\nc{\KA}{\on{K}^{\mathrm{A}}}
\nc{\KB}{\on{K}^{\mathrm{B}}}
\nc{\Res}{\on{Res}}
\nc{\Fc}[1][{n,m}]{\mathbf{F}_{#1}}
\nc{\tphi}{\tilde{\varphi}}
\nc{\CO}{\mathscr{O}}
\nc{\inte}{\mathrm{int}}
\nc{\Oint}{\mathcal{O}^{\ge0}_{\inte}}
\nc{\vs}{\vspace}
\nc{\tL}{\widetilde{L}}
\nc{\noi}{\noindent}
\nc{\heigh}{\mathfrak{t}}
\nc{\lowest}{\mathfrak{l}}

\newenvironment{red}{\relax\color{red}}{\hspace*{.5ex}\relax}
\newcommand{\ber}{\begin{red}}
\newcommand{\er}{\end{red}}

\newlength{\mylength}
\newenvironment{blue}{\relax\color{blue}}{\hspace*{.5ex}\relax}
\newcommand{\beb}{\begin{blue}}
\newcommand{\eb}{\end{blue}}

\title{A Categorification of $\displaystyle{\mathfrak q} (2)$-crystals}

\author[D. Grantcharov, J. H. Jung, S.-J. Kang,  M. Kim]{Dimitar Grantcharov$^{1}$, Ji Hye Jung$^{2}$, Seok-Jin
Kang$^{3}$, Myungho Kim}

\address{Department of Mathematics \\
         University of Texas at Arlington \\ Arlington, TX 76021, USA}

         \email{grandim@uta.edu}

\address{Department of Mathematical Sciences \\
         Seoul National University \\ Seoul 151-747, Korea}

         \email{jung.ji.hye@hotmail.com}

\address{Department of Mathematical Sciences
         and
         Research Institute of Mathematics \\
         Seoul National University \\ Seoul 151-747, Korea}

         \email{sjkang@snu.ac.kr}

\address{Department of Mathematics \\
          Kyung Hee University \\
           Seoul 02447, Korea}
           
           \email{mkim@khu.ac.kr}

\thanks{$^{1}$This work was partially supported by  NSA grant H98230-13-1-0245}
\thanks{$^{2}$This work was  supported by  NRF Grant \# 2014-021261 and
by NRF-2013R1A1A2063671}
\thanks{$^{3}$This work was supported by NRF Grant  \# 2014-021261 and by NRF Grant \# 2013-055408}

\keywords{{crystal bases}, {odd Kashiwara operators}, {quantum queer
superalgebras, }{$\displaystyle{\mathfrak q} (2)$-categorification}}
\subjclass[2010]{17B37, 81R50}

\begin{document}

\maketitle

\begin{abstract}
We provide a categorification of $\mathfrak{q}(2)$-crystals on the
singular  $\mathfrak{gl}_{n}$-category ${\mathcal O}_{n}$. Our
result extends the $\mathfrak {gl}_{2}$-crystal structure on ${\rm
Irr} ({\mathcal O}_{n})$ induced from the work of
Bernstein-Frenkel-Khovanov. Further properties of the ${\mathfrak
q}(2)$-crystal ${\rm Irr} ({\mathcal O}_{n})$ are also discussed.
\end{abstract}

\section*{Introduction}

The {\it crystal basis theory} is one of the most prominent
discoveries in the modern representation theory. Crystal bases,
which can be understood as {\it global bases} at $q=0$, have been
introduced by Kashiwara \cite{Kas90,Kas91,Kas93} and have many
significant applications to a wide variety of mathematical and
physical theories. In particular, their nice behavior with respect
to tensor products leads to elegant explanations of a lot of
combinatorial  phenomena  such as combinatorics of Young tableaux and
Young walls \cite{Kang2003, KN}. On the other hand, Lusztig took a
geometric approach to develop the {\it canonical basis theory} \cite{Lus90,Lus91},
which turned out to be deeply related to {\it categorification
theory} as is the case with global basis theory.

In \cite{BFK}, Bernstein, Frenkel and Khovanov discovered a close
connection between the singular $\mathfrak{gl}_{n}$-category
$\mathcal{O}_{n}$ and the $n$-fold tensor power ${\V}^{\otimes n}$,
where $\V$ is the 2-dimensional natural representation of
$\mathfrak{sl}_{2}$. Their result initiated the categorification
program of $\mathfrak{sl}_{2}$-representation theory, which was
extended to the quantum algebra $U_q(\mathfrak{sl}_2)$ \cite{Str}
and to general tensor products of finite-dimensional
$U_q(\mathfrak{sl}_2)$-modules \cite{FKS}. That is, they obtained
several versions of (weak) {\it $\mathfrak{sl}_{2}$-categorification} in
the sense of Chuang and Rouquier \cite{CR}.

In recent years, there is growing interest in the crystal basis
theory of the quantum superalgebras. A major accomplishment in this
direction is the development of crystal basis theory of the quantum
superalgebra $U_{q}(\mathfrak{gl}(m|n))$ for the tensor  modules; i.e., the
modules arising from tensor powers of the natural representation
\cite{BKK}. Such a theory was developed for the quantum
superalgebras $U_{q}(\mathfrak{q}(n))$ for the category of tensor
modules \cite{GJKKK, GJKKK2} and
 $U_{q}(\mathfrak {osp}(m|2n))$ for
a certain semisimple tensor category of $U_{q}(\mathfrak {osp}(m|2n))$-modules
\cite{Kw}.

The $\mathfrak{q}(n)$-case is especially interesting and challenging
both from algebraic and combinatorial perspectives. A definition of
$U_{q}(\mathfrak{q}(n))$ was first introduced in \cite{Ol} using the
Fadeev-Reshetikhin-Turaev formalism. In \cite{GJKK}, an equivalent
definition of $U_{q}(\mathfrak{q}(n))$ was given in the spirit of
Drinfeld-Jimbo presentation and the highest weight representation
theory was developed. Moreover, in \cite{GJKKK, GJKKK2, GJKKK3}, the
crystal basis theory for $U_{q}(\mathfrak {q}(n))$-modules was
established, which provides a representation theoretic
interpretation of combinatorics of semistandard decomposition
tableaux.

We now explain the main result of this paper. One important
consequence of $\mathfrak {sl}_{2}$-categorification in \cite{BFK}
is that the set ${\rm Irr}({\mathcal O}_{n})$ of isomorphism classes
of simple objects in ${\mathcal O}_{n}$ admits a $\mathfrak{
gl}_{2}$-crystal structure. The categorified Kashiwara operators
$\mathcal{E}$ and $\mathcal{F}$ are constructed using the
translation functors given by the $n$-dimensional natural
$\mathfrak{gl}_{n}$-module
$L(e_{1})$ and its dual $L(e_{1})^*$.

In the present paper, we investigate $\mathfrak{q}(2)$-crystal
structure on ${\rm Irr}({\mathcal O}_{n})$. We also use the
translation functors to construct the categorified odd Kashiwara
operators $\overline{\mathcal E}$ and $\overline{\mathcal F}$.
However, we use the infinite-dimensional irreducible highest weight
$\mathfrak{gl}_{n}$-module   $L(e_{n})$ with highest weight
 $e_{n}$
and its dual $L(e_n)^*$, which fits very
naturally in our setting. We believe our result is the first step
toward the ${\mathfrak q}(2)$-categorification theory and it will
generate  various interesting developments in categorical
representation theory of (quantum) superalgebras.

The organization of the paper is as follows. In the first  two
 sections, we collect some of basic definitions and
properties related to the $\mathfrak{gl}_2$-crystal structure on
${\rm Irr}({\mathcal O}_{n})$. The third section is devoted
to the properties of ${\mathfrak{q}} (2)$-crystals used in this
paper. The definition of categorified odd Kashiwara operators on
${\mathcal O}_{n}$, as well as the main result of this paper, are
included in Section 4. In the last section, we discuss
further properties of the $\mathfrak{q} (2)$-crystals related to
parabolic subcategories of ${\mathcal O}_{n}$.

\vskip 3mm

\noindent{\bf Acknowledgements.} We would like to thank V. Mazorchuk
for the fruitful discussions and for bringing our attention to the
paper \cite{Kah}. The first author would like to thank Seoul
National University for the warm hospitality and the excellent
working conditions.
\vskip 5mm


\section{The category  $\mathcal O_n$}

Let $\mathfrak{g} = \mathfrak{gl}_{n}$  $(n \ge 2)$ be the general linear Lie
algebra over the complex number field $\C$ with the triangular decomposition $\g = \mathfrak n_- \oplus
\mathfrak h \oplus \mathfrak n_+$. We denote by $U(\g)$ its
universal enveloping algebra and by $\mathcal Z (\g)$ the center of
$U(\g)$. Choose an orthonormal basis $\{e_1,\ldots,e_n\}$ of $\R^n$
and identify $\C \otimes_{\R} \R^n$ with $\mathfrak h^*$, the dual
of $\mathfrak h$. Thus $\Delta := \{e_i - e_j \mid i<j \}$ is the
set of positive roots and $\Pi:=\{ e_i - e_{i+1} \mid 1 \le i \le
n-1 \}$ is the set of simple roots. The Weyl group of $\mathfrak{g}$
is isomorphic to the symmetric group $S_{n}$, which acts on
${\mathfrak h}^*$ by permuting $e_{i}$'s.

We say that a $\g$-module $M$ is a {\it weight module} if  $$M =
\bigoplus_{\lambda \in {\mathfrak h}^*} M^{\lambda}, \ \
\text{where} \ \ M^{\lambda} = \{ m \in M \; | \; hm = \lambda(h)m \
\mbox{ for all } h \in {\mathfrak h} \}.$$ A linear functional
$\lambda \in {\mathfrak h}^*$ is called a {\it weight} of $M$ if
$M^{\lambda} \neq 0$. We denote by  $\text{Supp}(M)$  the
set of weights of $M$. Note that any weight of $M$ is a linear
combination of $e_{i}$'s. For a weight module $M= \bigoplus_{\lambda
\in {\mathfrak h}^*} M^{\lambda}$, let  $M^{*}: = \bigoplus_{\lambda
\in {\mathfrak h}^*} \Hom_{\mathbb C} (M^{\lambda}, {\mathbb C})$ be
the restricted dual of $M$ with the $\mathfrak {g}$-module action
given by
$$ (g f) (m) = f (-gm) \ \ \text{for} \ g \in \mathfrak{g}, f \in
M^*, m \in M.$$
Note that
$\text{Supp}(M^*) = - \text{Supp}(M)$

\vskip 3mm

Let $\V = \C v_{1} \oplus \C v_{2}$ be the 2-dimensional natural
representation of $\mathfrak{gl}_{2}$, where the
$\mathfrak{gl}_{2}$-action is given by left multiplication. Hence we
have $\text{wt}(v_1) = e_1$ and $\text{wt}(v_2) = e_2$. Recall that
the special linear Lie algebra $\mathfrak{sl}_{2}$ is the subalgebra
of $\mathfrak{gl}_{2}$ generated by
$$
E= \left(\begin{matrix} 0 & 1 \\ 0 & 0  \end{matrix}  \right), \quad
F= \left(\begin{matrix} 0 & 0 \\ 1 & 0  \end{matrix}  \right), \quad
H= \left(\begin{matrix} 1 & 0 \\ 0 & -1  \end{matrix}  \right).
$$
Thus its universal enveloping algebra $U(\mathfrak{sl}_{2})$ is the
associative $\C$-algebra generated by $E, F, H$ with defining
relations
\begin{align*}
  EF-FE=H, \quad HE-EH=2E, \quad HF-FH=-2F.
\end{align*}
The $\mathfrak{gl}_{2}$-action on $\V$ induces an
$\mathfrak{sl}_{2}$-action given by
\begin{equation*}
\begin{aligned}
& H v_1 = v_1, \quad E v_1 = 0, \quad F v_1 = v_2, \\
& H v_2 = -v_2, \quad E v_2 = v_1, \quad F v_2 =0.
\end{aligned}
\end{equation*}
It follows that the $\mathfrak{sl}_{2}$-weight of $v_{1}$ is $1$ and
that of $v_{2}$ is $-1$. For each $n\ge 2$, the tensor space
$\V^{\otimes n}$ admits a $U({\mathfrak sl}_{2})$-module structure
via the comultiplication $\Delta : U(\mathfrak{sl}_2) \rightarrow
U(\mathfrak{sl}_2) \otimes U(\mathfrak{sl}_2) $ given by
\begin{align*}
\Delta(E)=E \otimes 1 + 1\otimes E, \quad \Delta(F)=F \otimes 1 +
1\otimes F, \quad  \Delta(H)=H \otimes 1 + 1 \otimes H.
\end{align*}

\vskip 3mm

Let ${\mathcal W}  = {\mathcal W} (\g)$ be the category of all
weight modules $M$ such that $\dim M^{\lambda} < \infty$ for all
$\lambda \in {\mathfrak h}^*$. We denote by $\mathcal O=\mathcal
O(\g)$ the full subcategory of ${\mathcal W} (\g)$ consisting  of
finitely generated $U(\g)$-modules that are  locally $U(\mathfrak
n_+)$-nilpotent. The category $\mathcal O$ is  known as the {\it
Bernstein-Gelfand-Gelfand category}.
For a detailed exposition of the category $\mathcal O$, see for example,
\cite{H}.

Let
$$\rho =\dfrac{1}{2} \sum_{i<j} (e_i - e_j) = \dfrac{n-1}{2}e_1 + \dfrac{n-3}{2}e_2 + \cdots + \dfrac{1-n}{2}e_n,$$
the half sum of positive roots. For a sequence $a_1, \ldots, a_n$ of
$1$'s and $2$'s, we denote by $M(a_1, \ldots, a_n)$ and  $L(a_1,
\ldots, a_n)$ the Verma module with highest weight $a_1 e_1 + \cdots
+ a_n e_n -\rho$ and its simple quotient, respectively.

For each $i \in \Z$, define  $\mathcal O_{i,n-i}$ to be the full
subcategory of $\mathcal O$ consisting of
$\mathfrak{gl}_{n}$-modules $M$ whose composition factors are of the
form $L(a_1, \ldots, a_n)$  with exactly $i$-many $2$'s. The
category ${\mathcal O}_{i, n-i}$ is a singular block of $\mathcal O$
corresponding to the subgroup $S_{i} \times S_{n-i}$ of $S_{n}$. For
$i <0$ or $i >n$, ${\mathcal O}_{i, n-i}$ consists of the zero
object only. We define
\begin{equation}\label{eq:On}
\mathcal O_n\seteq \soplus_{i=0}^{n} \  \mathcal O_{i,n-i},
\end{equation}
the main category  of our interest. We denote $G(\mathcal O_n):= \C
\otimes K(\mathcal O_n)$, where $K(\mathcal O_n)$ is the
Grothendieck group of $\mathcal O_n$. As usual, we write $[M]$ for
the isomorphism class of an object $M$ in $\mathcal O_n$.

An alternative description of ${\mathcal O}_{i, n-i}$ is given as
follows. Let $\chi: {\mathcal Z}(\mathfrak{g}) \rightarrow \C$ be an
algebra homomorphism. We define ${\mathcal O}_{\chi}$ to be the
subcategory of $\mathcal O$ consisting of $\mathfrak{g}$-modules $M$
such that for each $z \in {\mathcal Z}(\mathfrak{g})$ and $m \in M$,
we have $(z - \chi(z))^k \, m =0$ for some $k>0$. Then we get the
 {\it central character decomposition}  $${\mathcal O}=
\bigoplus_{\chi\in {\mathcal Z}(\mathfrak{g})^{\vee}} {\mathcal
O}_{\chi},$$ where ${\mathcal Z}(\mathfrak{g})^{\vee}$ denotes the
set of all algebra homomorphisms ${\mathcal Z}(\mathfrak{g})
\rightarrow \C$. Note that ${\mathcal Z}(\mathfrak{g})$ acts on a
highest weight module with highest weight $\lambda$ by a constant
$\chi^{\lambda}(z)$ $(z \in {\mathcal Z}(\mathfrak{g}))$. We will
write ${\mathcal O}_{\lambda}$ for ${\mathcal O}_{\chi^{\lambda}}$.

On the other hand, for each $\nu \in {\mathfrak h}^*$, set
$\overline{\nu}=\nu + \Z \Delta \in {\mathfrak h}^* \big/ \Z
\Delta$, where $\Z \Delta$ denotes the root lattice of $\mathfrak g$.
 We define ${\mathcal O}[\overline{\nu}]$ to be the full
subcategory of $\mathcal O$ consisting of
$\mathfrak{gl}_{n}$-modules $M$ such that $\text{wt}(M) \subset
 \overline{\nu}$.  Then we have
 {\it the support decomposition}
$${\mathcal O} = \bigoplus_{\overline{\nu} \in {\mathfrak h}^* \big/ \Z
\Delta} {\mathcal O}[\overline{\nu}].$$
 The category ${\mathcal
O}_{i,n-i}$ coincides with ${\mathcal O}_{\omega_i - \rho}$, and
${\mathcal O}_{\omega_i - \rho}$ is a full subcategory of
${\mathcal O}[\overline{\omega_i - \rho}]$, where
$$\omega_i := 2 \sum_{j=1}^{i} e_j + \sum_{j=i+1}^{n} e_j$$
is the shifted $i$-th fundamental weight.

Similarly, the category $\mathcal W$ has the central character
decomposition and the support decomposition
$${\mathcal W}= \bigoplus_{\chi\in
{\mathcal Z}(\mathfrak{g})^{\vee}} {\mathcal W}_{\chi} =
\bigoplus_{\overline{\nu} \in {\mathfrak h}^* \big/ \Z \Delta}
{\mathcal W}[\overline{\nu}].$$ Note that $\mathcal W$ has the
central character decomposition due to the fact that the ${\mathcal
Z}(\mathfrak{g})$-action is stable on each weight space. Set
$${\mathcal W}_{\lambda}:={\mathcal W}_{\chi^{\lambda}}, \quad
{\mathcal W}_{i, n-i}:={\mathcal W}_{\omega_i - \rho} \cap {\mathcal
W}[\overline{\omega_i - \rho}], \quad {\mathcal W}_{n}:=
\bigoplus_{i=0}^{n} {\mathcal W}_{i, n-i}.$$

For each $0 \le i \le n$, let ${\rm pr}_{i}: {\mathcal W}
\rightarrow \, {\mathcal W}_{i,n-i}$ be the canonical projection.
Clearly, ${\rm pr}_i (\mathcal O ) = {\mathcal O_{i,n-i}}$.
 Following \cite{BFK},
we define
\begin{align}
\mathcal E_i :  {\mathcal O}_{i,n-i} \to  {\mathcal O}_{i+1,n-i-1}, \; \mathcal E_i := {\rm pr}_{i+1} \circ  \big(- \otimes L(e_1) \big), \\
\mathcal F_i :  {\mathcal O}_{i,n-i} \to  {\mathcal O}_{i-1,n-i+1}, \; \mathcal F_i := {\rm pr}_{i-1} \circ  \big(- \otimes L(e_1)^* \big),
\end{align}
where  $L(e_1)$ is the $n$-dimensional natural representation of
$\g$. Now we define the  exact endofunctors  ${\mathcal E}$ and ${\mathcal
F}$ on  $\mathcal O_n$ by
\begin{align}
\mathcal E := \soplus_{i=0}^{n} \mathcal E_i, \quad
\mathcal F := \soplus_{i=0}^{n} \mathcal F_i.
\end{align}
 We denote by $[\mathcal E]$ and $[\mathcal F]$ the linear endomorphisms on $G(\mathcal O_n)$ induced from the functors $\mathcal E$ and $\mathcal F$, respectively.

The following theorem plays an important role in this paper.

\begin{theorem}  \rm (\cite{BFK}) \label{sl2}
  \begin{enumerate}
\item $(\mathcal E, \mathcal F)$ is a biadjoint pair.


\item The correspondence  $E \mapsto [\mathcal E]$,
$F \mapsto [\mathcal F]$ defines a $U(\mathfrak{sl}_2)$-action  on
$G(\mathcal O_n)$.

\item
 The simple objects in $\mathcal O_n$ correspond to  weight  vectors  in $G(\mathcal O_n)$.

 \item
   There is a $U(\mathfrak{sl}_2)$-module isomorphism
\begin{equation}
\begin{array}{ccc}
\Upsilon \ : \ G(\mathcal O_n) & \rightarrow & \V^{\otimes n} \\
 \left[M(a_1, \ldots, a_n)\right]  &\mapsto & v_{a'_1} \otimes \cdots \otimes v_{a'_n},
\end{array}
 \end{equation}
 where $1':=2$ and $2':=1$.
      \end{enumerate}
\end{theorem}

\begin{theorem}
\rm(\cite[\S 7.4.3]{CR})
The category $\mathcal O_n$ provides a (strong) $\mathfrak{sl}_2$-categorification in the sense of Chuang-Rouquier.

\end{theorem}

\vskip 5mm


\section{$\mathfrak{gl}_2$-crystal structure on ${\rm Irr} (\mathcal
O_n)$}

In this section, we will discuss the $\mathfrak{gl}_2$-crystal
structure on ${\rm Irr} (\mathcal O_n)$, the set of isomorphism
classes of simple objects in ${\mathcal O_n}$. We first recall the
definition of  $\mathfrak{gl}_{2}$-crystal.
For details, see for example, \cite{HK2002}.

Set $P:=\Z e_1 \oplus \Z e_2$ and $\alpha_1:=e_1-e_2$.
 Let $(k_1, k_2)$ be the  basis of $P^*$
which is  dual to $(e_1,e_2)$.
The natural pairing $P^* \times P \to \Z$ is denoted by
$\langle \, , \, \rangle$.

\begin{definition}
{\rm  An {\em (abstract) $\mathfrak{gl}_2$-crystal} is a set $B$
together with the maps $\te, \tf \cl B \to B \sqcup \{0\}$, $\vphi,
\eps \cl B \to \Z \sqcup \{-\infty\}$, and $\wt\cl B \to P$
satisfying the following conditions (see \cite{Kas93}):

\begin{itemize}
\item[(i)] $\wt(\te b) = \wt b + \alpha_1$ if $\te b \neq 0$,

\item[(ii)] $\wt(\tf b) = \wt b - \alpha_1$ if $\tf b \neq 0$,

\item[(iii)] for any $b\in B$, $\vphi(b) = \eps(b) +$
$\langle k_1-k_2, \wt b\rangle$,

\item[(iv)] for any $b,b'\in B$,
$\tf b = b'$ if and only if $b = \te b'$,

\item[(v)] for any $b\in B$
such that $\te b \neq 0$, we have  $\eps(\te b) = \eps(b) - 1$,
$\vphi(\te b) = \vphi(b) + 1$,

\item[(vi)] for any $b\in B$ such that $\tf b \neq 0$,
we have $\eps(\tf b) = \eps(b) + 1$, $\vphi(\tf b) = \vphi(b) - 1$,

\item[(vii)] for any $b\in B$ such that $\vphi(b) = -\infty$, we
have $\te b = \tf b = 0$.
\end{itemize}
}
\end{definition}
\vskip 3mm

For each object $S \in \ \mathcal O_n $, set
\begin{align*}
\vphi(S) \seteq \max \set{m \in \Z_{\ge 0}}{\mathcal F^m (S)\neq 0}, \quad
\varepsilon(S) \seteq \max \set{m \in \Z_{\ge 0}}{\mathcal E^m (S)\neq 0}.
\end{align*}

The  $\mathfrak{sl}_2$-categorification on $\mathcal O_n$ has the
following nice properties.

\begin{prop} \rm (\cite[Proposition 5.20]{CR}, \cite[Proposition
2.3]{Losev}) \hfill

 Let $S$ be a simple object in $\mathcal O_n$
with $\varepsilon(S) \neq 0$ (respectively, $\vphi(S) \neq 0$).
  \begin{enumerate}
    \item The object $\mathcal E (S) $ (respectively, $\mathcal F(S)$) has simple socle
    and simple head, and they are isomorphic to each other.
     \item For any other subquotient $S'$ of $\mathcal E (S) $ (respectively, $\mathcal F(S)$),
     we have  $\varepsilon(S') \le \varepsilon(S)-1$ (respectively, $\vphi(S') \le \vphi(S)-1$).
  \end{enumerate}
\end{prop}

\begin{remark}
In \cite{BrSt}, Brundan and Stroppel investigated  intensively
 the parabolic analogue of  the BGG category $\mathcal O$ associated with the subalgebra $\gl_m \oplus \gl_n$ of $\gl_{m+n}$.
They showed that a sum  ${\mathcal O}(m,n,\Z_{\ge 1})$ of its integral blocks forms an integrable representation of the 2-Kac-Moody algebra ${\mathfrak U}(\mathfrak{sl}_{\Z_{\ge 1}})$ in the sense of \cite{R} (\cite[Remark 5.7]{BrSt}).
 Moreover, in that case, they described the crystal structure on the set of simple objects of ${\mathcal O}(m,n,\Z_{\ge 1})$ in an explicit form (\cite[(3.15)]{BrSt}) and derived an analogue of the above proposition explicitly (\cite[Lemma 4.9]{BrSt}).
 \end{remark}

For a simple object $S$ in ${\mathcal O}_{n}$, let $\mathsf{wt}([S])
\in \Z$ be the $\mathfrak{sl}_{2}$-weight of $[S]$ in  $G(\mathcal
O_n)$.
Define
\begin{align*}
\te ([S]) := [{\rm hd}\, \mathcal E(S)], \quad \tf ([S]) := [{\rm
hd}\, \mathcal F(S)].
\end{align*}
Since the head and socle of $\mathcal{E}(S)$ are isomorphic, we may
define $\te([S]) = [{\rm soc}({\mathcal E}(S))]$ and similarly for
$\tf([S])$.

Then $\big({\rm Irr} (\mathcal O_n), \mathsf{wt}, \vphi,
\varepsilon, \te, \tf \big)$ becomes an $\mathfrak{sl}_2$-crystal
(see the last paragraph of \cite[\S 2.4]{Losev}). For example, if
${\rm hd} \, \mathcal E(S) \cong S'$, then we have
$$0 \neq \Hom_{\mathcal O_n}(\mathcal E(S), S') \cong \Hom_{\mathcal O_n}( S, \mathcal F(S')). $$
Thus $S$ is a simple submodule of $ \mathcal F(S')$ so that we have
$S \cong {\rm soc}\,  \mathcal F(S')$ by the above proposition. That
is, if $\te([S])=[S']$, then $[S]=\tf([S'])$ as desired.

Note that, by the $U(\mathfrak{sl}_{2})$-module isomorphism in
Theorem \ref{sl2}(4), the $\mathfrak{sl}_{2}$-weight of $[L(a_1,
\ldots, a_n)]$ is given by
$$\mathsf{wt}([L(a_1, \ldots, a_n)])=\sharp \{ i \; | \; a_i=2 \} - \sharp \{ i \; | \; a_i=1 \}.$$
Hence by setting
$$\wt ([ L(a_1,\ldots,a_n) ]) := (\sharp \{ i \; | \; a_i=2 \}) e_1 + (\sharp \{ i \; | \; a_i=1 \}) e_2,$$
$\big({\rm Irr} (\mathcal O_n),\wt, \vphi, \varepsilon, \te, \tf
\big)$ becomes a $\mathfrak{gl}_2$-crystal.

\vskip 3mm

Let $\B=\{b_1, b_2\}$ be the $\mathfrak{sl}_2$-crystal of $\V$. By
defining $\text{wt}(b_1)=e_1$, $\text{wt}(b_2)=e_2$, $\B$ becomes a
$\mathfrak{gl}_{2}$-crystal.
 Recall that the {\em tensor product rule for $\mathfrak{gl}_2$-crystals} gives a  $\mathfrak{gl}_2$-crystal structure on $\B^{\otimes n}= \B \times \cdots \times \B$ (see, for example,  \eqref{eq1:tensor product}).  The following theorem describes
the $\mathfrak{gl}_{2}$-crystal structure on
$\text{Irr}(\mathcal O_n)$.

\begin{theorem}  \rm  \label{thm:gl2-crystal} (\cite[Theorem 4.4]{BrKl})

As a $\mathfrak{gl}_2$-crystal,  $\big({\rm Irr} (\mathcal O_n),\wt,
\vphi, \varepsilon, \te, \tf \big)$ is isomorphic to $\B^{\otimes
n}$  under the map
$$[L(a_1, \ldots, a_n) ] \mapsto b_{a'_1} \otimes \cdots \otimes b_{a'_n}.$$
\end{theorem}

\begin{remark}
Note that the functors $e_1$ and $f_1$ defined in \cite{BrKl}
correspond to our functors  $\mathcal F$ and $\mathcal E$,
respectively.
If we take the full subgraph of the $\mathfrak{gl_\infty}$-crystal $\Z^n$ given in \cite{BrKl} with vertices   $\widetilde \B:=\set{(a_1, \ldots, a_n)}{a_i \in \{1,2\} } \subset \Z^n$, then $\widetilde \B$  can be regarded as a $\mathfrak{gl}_2$-crystal in a natural way.
Remark that in \cite{BrKl}, the opposite tensor product rule for
$\mathfrak{gl}_2$-crystals was used.
The map $\psi : (a_1, \ldots, a_n) \mapsto  (b_{a'_1}, \ldots, b_{a'_n}) $  becomes a bijection
between $\widetilde \B$ and $\B$ satisfying $\psi(\tf_1(a_1, \ldots, a_n)) = \te(\psi(a_1, \ldots, a_n)) $ and $\psi(\te_1(a_1, \ldots, a_n)) = \tf(\psi(a_1, \ldots, a_n)) $.
\end{remark}
\vskip 5mm


\section{$\mathfrak{q} (2)$-crystals} \label{sec:q2crystal}

In this section we recall  the definition of  $\mathfrak{q}
(2)$-crystal and provide a description of the connected components
of $\B^{\otimes n}$ as $\mathfrak{q} (2)$-crystals. The notion of
{\it abstract $\mathfrak{q}(n)$-crystal} and the {\it queer tensor
product rule} are introduced in \cite{GJKKK, GJKKK2, GJKKK3}. In
this paper, we consider $\mathfrak{q}(2)$-crystals only.

Recall that $\mathfrak{q} (n)$ is the Lie subalgebra of the general linear Lie superalgebra  $\mathfrak{gl} (n|n)$ over $\C$ consisting of all matrices of the form $\left( \begin{matrix}  A & B \\ B & A  \end{matrix}  \right)$. The even part of $\mathfrak{q}  (n)$ is naturally isomorphic to $\mathfrak{gl} _n$. The structure theory of  $\mathfrak{q} (n)$ is rather different from the one of the other classical Lie superalgebras. For more details on the properties of $\mathfrak{q} (n)$ we refer the reader for example to \cite{GJKK, Ol, Ser}.

\begin{definition}
{\rm An {\em $\mathfrak{q} (2)$-crystal} is  a
$\mathfrak{gl}_2$-crystal together with the maps $\teone, \tfone\cl
B \to B \sqcup \{0\}$ satisfying the following conditions:

\begin{itemize}
\item[(i)] $\wt(B)\subset P^{\ge0} \seteq \Z_{\ge 0} e_1 \oplus \Z_{\ge 0} e_2$,

\item[(ii)]  $\wt(\teone b) = \wt(b) + \alpha_1$, $\wt(\tfone b) = \wt(b) -
\alpha_1$,

\item[(iii)]  for all $b, b' \in B$, $\tfone b = b'$ if and only if $b = \teone b'$.
\end{itemize}
}
\end{definition}
Note that in  \cite{GJKKK, GJKKK2}, the $\mathfrak{gl}_2$-crystals satisfying the above conditions are called {\it abstract} $\mathfrak{q}(2)$-crystals. In this paper, we simply call them $\mathfrak q(2)$-crystals.

Let  $B$ be a $\mathfrak{q}(2)$-crystal (respectively, a $\mathfrak{gl}_2$-crystal) and let $B'$ be a subset of $B$.
We say that  $B'$ is a {\it $\mathfrak{q}(2)$-subcrystal} (respectively, {\it $\mathfrak{gl}_2$-subcrystal}) of $B$, if   $x(b) \in B'\sqcup\{0\}$ for every $b \in B'$ and $x = \te,\tf, \teone, \tfone$ (respectively, $x=\te,\tf$).

The queer tensor product rule is given in the following theorem.

\Th  \rm  \cite{GJKKK, GJKKK2, GJKKK3} \hfill

Let $B_1$ and $B_2$ be  $\mathfrak{q} (2)$-crystals. Define
the {\em tensor product} $B_1 \otimes B_2$ of $B_1$ and $B_2$ to be
$(B_1 \times B_2, \wt, \vphi, \varepsilon, \te,\tf, \teone, \tfone
)$, where
\begin{align*}
&\wt(b_1\otimes b_2) = \wt(b_1) +\wt(b_2), \\
&\varepsilon(b_1\otimes b_2) =
\max \{ \varepsilon(b_1) - \vphi(b_1) + \varepsilon(b_2), \ \varepsilon(b_1)\}, \\
&\vphi(b_1\otimes b_2) =
\max \{ \vphi(b_1) - \eps(b_2)+\vphi(b_2), \ \vphi(b_2)\},
\end{align*}
and
\begin{equation} \label{eq1:tensor product}
\begin{aligned}
\te (b_1 \otimes b_2) & = \begin{cases} \te b_1 \otimes b_2 \ &
\text{if} \ \vphi(b_1) \ge \eps(b_2), \\
b_1 \otimes \te b_2 \ & \text{if} \ \vphi(b_1) < \eps(b_2),
\end{cases} \\
\tf (b_1 \otimes b_2) & = \begin{cases} \tf b_1 \otimes b_2 \
& \text{if} \  \vphi (b_1) > \eps (b_2), \\
b_1 \otimes \tf b_2 \ & \text{if} \ \vphi (b_1) \le \eps (b_2),
\end{cases}
\end{aligned}
\end{equation}

\begin{equation} \label{eq2:tensor product}
\begin{aligned}
\teone (b_1 \otimes b_2) & = \begin{cases} \teone b_1 \otimes b_2
& \text{if $\lan k_1, \wt b_2 \ran =\lan k_2, \wt b_2 \ran =0$,} \\
b_1 \otimes \teone b_2 &  \text{otherwise,}
\end{cases} \\
\tfone(b_1 \otimes b_2) & = \begin{cases} \tfone b_1 \otimes b_2
& \text{if $\lan k_1, \wt b_2 \ran = \lan k_2, \wt b_2 \ran =0$,}
 \\
b_1 \otimes \tfone b_2   
& \text{otherwise}.
\end{cases}
\end{aligned}
\end{equation}

 Then $B_1 \otimes B_2$  is
a $\q (2)$-crystal.
\enth

For a given  $\q (2)$-crystal, we draw an arrow $\xymatrix{b
\ar@{->}@<0.2ex>[r]^-{1} &b'}$ if and only if $\tf(b) =b' $ and draw
an arrow $\xymatrix{b \ar@{-->}@<0.2ex>[r]^-{\bar 1} &b'}$ if and
only if $\tfone(b) =b' $. The resulting oriented graph is called a
\emph{$\mathfrak{q}(2)$-crystal graph}.

 For a vertex $b$ in a  $\mathfrak q(2)$-crystal graph $B$, we denote by $C(b)$
the connected component of $b$ in $B$.
The connected component as a $\mathfrak{gl}_2$-crystal will be denoted by  $C_{\mathfrak{gl}_2}(b)$.

An element $b$ in a $\mathfrak q(2)$-crystal (respectively, $\mathfrak{gl}_2$-crystal) is called a {\it highest weight vector} (respectively, {\it $\mathfrak{gl}_2$ -highest weight vector}) if $\teone b=\te b=0$ (respectively, $\te b=0$).
If $\varphi(b)=0$  and $\te^{\eps(b)} b$ is a highest weight vector, then we call $b$ a {\it lowest weight vector}.

\begin{example} \label{ex:BtensorN} \hfill

\begin{enumerate}
\item Let $\mathbf{B}=\{ b_1, b_2   \}$ be the $\mathfrak{gl}_{2}$-crystal of $\V$.
Define
$$\teone(b_1)=0, \ \ \tfone(b_1)=b_2, \quad \teone(b_2)=b_1, \ \
\tfone(b_2)=0.$$ Then $\B$ is a $\mathfrak{q}(2)$-crystal
with $\mathfrak{q}(2)$-crystal graph
$$\xymatrix@C=5ex
{*+{1} \ar@<0.1ex>[r]^-{1} \ar@{-->}@<-0.9ex>[r]_{\ol 1} & *+{2} }$$
From now on, $b_1$ and $b_2$ are identified with $1$ and $2$, respectively.

\vskip 3mm

\item
By the queer tensor product rule, $\mathbf{B}^{\otimes  r }$ is a $\mathfrak{q} (2)$-crystal. The $\q (2)$-crystal structure
of $\B^{\otimes 4}$ is given below.

$$\hskip3em 
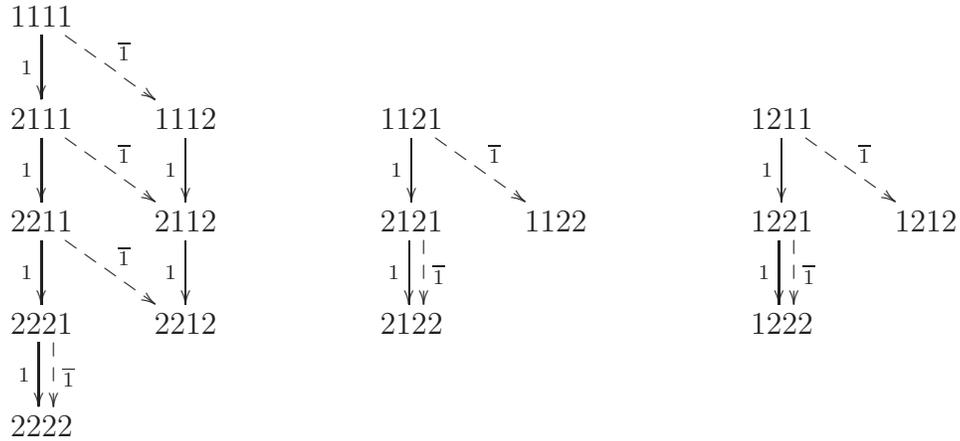
\xymatrix{ 1111\ar[d]_-{1} \ar@{-->}[dr]^{\ol 1} &  && & && & \\
2111\ar[d]_-{1} \ar@{-->}[dr]^{\ol 1}& 1112\ar[d]_-{1}&& 1121\ar[d]_-{1} \ar@{-->}[dr]^{\ol 1}&&& 1211\ar[d]_-{1} \ar@{-->}[dr]^{\ol 1}&\\
2211\ar[d]_-{1} \ar@{-->}[dr]^{\ol 1}&2112\ar[d]_-{1}&& 2121\ar@<-0.2ex>[d]_-{1} \ar@{-->}@<0.9ex>[d]^{\ol 1}& 1122&& 1221 \ar@<-0.2ex>[d]_-{1} \ar@{-->}@<0.9ex>[d]^{\ol 1} & 1212 \\
2221\ar@<-0.2ex>[d]_-{1} \ar@{-->}@<0.9ex>[d]^{\ol 1} &2212&& 2122 &&& 1222&\\
2222& && &&&& }$$
 Here we identify a sequence $a_1 \cdots a_r$ ($a_i \in \{1,2\}$) with the element $a_1 \otimes \cdots \otimes a_r \in \B^{\otimes r}$.
\vskip 3mm

\item
The connected component $C(22122122) \subset \B^{\otimes 8}$ is given below:

$$
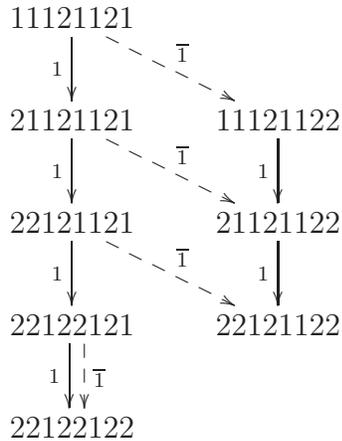
\xymatrix{
11121121 \ar[d]_-{1}  \ar@{-->}[dr]^{\ol 1} & \\
21121121\ar[d]_-{1}\ar@{-->}[dr]^{\ol 1} & 11121122 \ar[d]_-{1}\\
22121121\ar[d]_-{1}\ar@{-->}[dr]^{\ol 1}& 21121122 \ar[d]_-{1}\\
22122121\ar@<-0.2ex>[d]_-{1} \ar@{-->}@<0.9ex>[d]^{\ol 1}& 22121122 \\
22122122&
}$$

\end{enumerate}
\end{example}

\vskip 3mm

In Example \ref{ex:BtensorN}(2), we can observe the following decompositions of $\mathfrak{gl}_2$-crystals. 
\begin{prop}  \rm \label{prop-easy-dec}

For $r \ge 2$, the connected component $C(1^r)$ in $\B^{\otimes r}$
 is decomposed into
$$C(1^r)=C_{\mathfrak{gl}_2}(1^r) \sqcup C_{\mathfrak{gl}_2}(1^{r-1} 2) \cong C_{\mathfrak{gl}_2}(1^{r-1}) \otimes \B_{\mathfrak{gl}_2},$$
as $\mathfrak{gl}_2$-crystals.
\end{prop}

\begin{proof}
 Let $b \in \B^{\otimes r}$. It is not difficult to see that $\tfone \tf^x \tfone b =0$ for all $x \in \Z_{ \ge 0}$.
Note that $1^r$ is the only vector in $C(1^r)$ annihilated by  $\te$ and $\teone$ by \cite[Theorem 4.6(b)]{GJKKK2}.
Hence, an element of $C(1^r) \sqcup \{0\}$ is one of the form
$$\tf^x(1^r), \  \tf^y \tfone \tf^x(1^r),\ \ (x,y \in \Z_{\ge 0}).  $$
Clearly, $\tf^x(1^r) \in C_{\mathfrak{gl}_2}(1^r) \sqcup \{0\}$ and
$\tf^x\tfone(1^r) \in C_{\mathfrak{gl}_2}(1^{r-1} 2) \sqcup \{0\}$.
By direct calculations, we have

\begin{align*}
  \tfone \tf^x(1^r)=\begin{cases}
    2^x 1^{r-1-x} 2  = \tf^x \tfone(1^r)   \in C_{\mathfrak{gl}_2}(1^{r-1}2) \ \ &\text{if} \ \  0 \le x \le r-2,\\
    2^x 2 = \tf^r (1^r)   \in C_{\mathfrak{gl}_2}(1^{r}) \ \ &\text{if} \ \ x=r-1,\\
    0 \ \   &\text{otherwise}.
  \end{cases}
\end{align*}
Then it is clear that
 $\tf^y \tfone \tf^x(1^r) \in C_{\mathfrak{gl}_2}(1^r) \sqcup C_{\mathfrak{gl}_2}(1^{r-1} 2) \sqcup \{0\} $.
Hence,
$C(1^r) \subseteq C_{\mathfrak{gl}_2}(1^r) \ \sqcup \ C_{\mathfrak{gl}_2}(1^{r-1} 2).$
 Since $1^r, 1^{r-1}2=\tfone(1^r) \in C(1^r),$
 it follows that $C_{\mathfrak{gl}_2}(1^r) \sqcup C_{\mathfrak{gl}_2}(1^{r-1} 2)  = C(1^r)$.

Now we show $C_{\mathfrak{gl}_2}(1^r) \sqcup
C_{\mathfrak{gl}_2}(1^{r-1} 2) \cong C_{\mathfrak{gl}_2}(1^{r-1})
\otimes \B_{\mathfrak{gl}_2}$. We can regard
$C_{\mathfrak{gl}_2}(1^{r-1}) \otimes \B_{\mathfrak{gl}_2}$ as a
$\gl_2$-subcrystal of $\B_{\gl_2}^{\otimes r}$. Note that $b$ is a
$\gl_2$-highest weight vector in $\B_{\gl_2}^{\otimes r}$ if and
only if $b$ is a lattice permutation. Since
$C_{\gl_2}(1^{r-1})=\set{ 2^x 1^{r-1-x}}{ 0 \le x \le r-1},$ there
are only two $\gl_2$-highest weight vectors   in $C_{\gl_2}(1^{r-1})
\otimes \B_{\gl_2}$; $1^r$ and $1^{r-1}2$. Hence,
$C_{\mathfrak{gl}_2}(1^{r-1}) \otimes \B_{\mathfrak{gl}_2} =
C_{\mathfrak{gl}_2}(1^r) \sqcup C_{\mathfrak{gl}_2}(1^{r-1} 2).$
\end{proof}

\vskip 3mm

Recall that a finite sequence of positive integers $x=x_1 \cdots
x_N$ is called a {\em strict reverse lattice permutation} if for $1
\le k \le N$ and $2 \le i \le n$, the number of occurrences of $i$
is strictly greater than the number of occurrences of $i-1$ in $x_k
\cdots x_N$ as  long  as $i-1$ appears in $x_k  \cdots  x_N$
\cite{GJKKK3}.

\begin{prop}  \rm (\cite{GJKKK3}) \label{cor_strict reverse lattice permutation}
 An element $b_1 \otimes \cdots \otimes b_N \in \B^{\otimes N}$ is a
lowest weight vector if and only if it is a strict reverse lattice
permutation.
\end{prop}

We say that a sequence  consisting of 1's and 2's is a {\it trivial
lattice permutation} if

\ \ \, (i) the number of 1's and the number of 2's are the same,

\ \ (ii) in every proper  initial part, the number of occurrences of
1 is strictly larger than the number of occurrences of 2.

\vskip 3mm

 For a sequence $u$ in $\{1,2\}$, we denote by $|u|$ the length of $u$.
\begin{prop}  \rm \label{prop-c-l} \hfill

\begin{enumerate}

\item  Let $\ell = a_1 a_2 \cdots a_r$ be a $\q (2)$-lowest weight
vector in $\B^{\otimes r}$. Then there is a unique way to decompose
$\ell$ into the form
$$\ell = u_1 u_2 \cdots u_s 2$$
such that every $u_i$ is a trivial lattice permutation or a
maximal subsequence consisting of 2's only.

\item Let $A_{\ell}$ be the set of positive integers $k$ with $1 \le
k \le r-1$ such that
$$|u_1|+|u_2|+\cdots +|u_{i-1}| < k \le
|u_1|+|u_2|+\cdots +|u_i|,$$ where $u_i$ is a trivial lattice
permutation. For $b = b_1 \cdots b_r  \in \B^{\otimes r} $, define
$\widehat b$ to be the  sequence obtained  from $b$ by removing all
$b_k$'s for $k \in A_\ell$. We also define $\overline {b}$ to be the
subsequence $b_{k_1} b_{k_2} \cdots b_{k_m}$ of $b$, where
$A_{\ell}=\{ k_1 < k_2 < \cdots <k_m \}$.

\vskip 3mm \noindent
Then we have

\vskip 3mm

\begin{enumerate}
\item
$C(\ell) = \set{b \in \B^{\otimes r}}{ \widehat b \in C(\widehat
\ell), \ \ {\overline b}={\overline \ell}}$.

\vskip 2mm

\item The map $C(\ell) \to C(\widehat \ell)$ given by $b \mapsto \widehat
b$ is a bijection that commutes with  $\te$, $\tf$, $\teone$,
$\tfone$.
\end{enumerate}

 \end{enumerate}
\end{prop}

\Proof

Since $\ell$ is a $\q (2)$-lowest weight vector, it is a strict reverse lattice permutation by Corollary
 \ref{cor_strict reverse lattice permutation}.
In particular, we have $a_r=2$.
If $\ell=2^r$, we have $u_1=2^{r-1}.$
If $\ell \neq 2^r$, let $a_j$ be the  \emph{leftmost} $1$  that occurs  in $\ell$.
By the definition, $a_j a_{j+1} \cdots a_r$ is also a strict reverse lattice permutation,
therefore,
the number of occurrences of $2$ is strictly
greater than the number of occurrences of $1$ in $a_j a_{j+1} \cdots a_r$.
Hence, there is the smallest $k$ such that $j+1 \le k \le r-1$ and
the number of occurrences of $2$ is equal to
the number of occurrences of $1$ in $a_j a_{j+1} \cdots a_k$.
We let $u_1=a_1 \cdots a_{j-1}=2^{j-1}$, $u_2=a_{j} a_{j+1}\cdots a_k$ when $j \ge 2$, and
$u_1=a_{j} a_{j+1} \cdots a_k$ when $j=1$.
Since $k$ is the smallest one and the number of occurrences of $2$ is equal to
the number of occurrences of $1$ in $a_j a_{j+1} \cdots a_k$,
the subsequence $a_j a_{j+1} \cdots a_k$ is a trivial lattice permutation.
Since $a_{k+1} \cdots a_r$ is also a strict reverse lattice permutation,
we repeat the above procedure.
By the construction, it is straightforward that
the decomposition of $\ell$ into the form $\ell=u_1 u_2 \cdots u_s 2$
is unique.

Let $M:= \set{b \in \B^{\otimes r}}{ \widehat b \in C(\widehat
\ell), \ \overline b= \overline \ell}$. By defining  $\widehat 0
\seteq 0$, we obtain a bijection between $M \sqcup \{0\}$ and
$C(\widehat \ell) \sqcup \{0\}$ given by $b \mapsto \widehat b$. We
will show that this bijection commutes with $\te$, $\tf$, $\teone$
and $\tfone$.

Note that $\tfone, \teone$ act only on $b_r$ for $b \in \B^{\otimes r}$.
In addition, we have $r \not \in A_{\ell}$ so that $\widehat b = ub_r$ for some $u$.
It follows that
$$\widehat{\tfone(b)}=\tfone(\widehat{b}), \quad  \widehat{\teone(b)}=\teone(\widehat{b}).  $$

We know that
$$\vphi(b)=\max \set{k \ge 0}{\tf(b) \in \B^{\otimes r}}
 \text{  and  } \, \eps(b) =\max \set{k \ge 0}{\te(b) \in \B^{\otimes r}}.$$
Since $\overline b=\overline \ell$ is a sequence of trivial lattice permutations,
we have $\vphi(b)=\vphi(\widehat{b})$ and $\eps(b)=\eps(\widehat{b})$.
In particular, we have $\tf(\widehat{b})=0$ if and only if $\tf(b)=0$, and
 $\te(\widehat b)=0$ if and only if $\te(b)=0$.

   Assume that $\tf(b)\neq 0$. Then we have
  $\tf(b)=b_1 \cdots \tf(b_t) \cdots b_r$ for some $1 \le t \le r$.
  Since $\vphi(u) = \eps(u)=0 $ for every trivial lattice permutation $u$,
  the tensor product rule implies that
  $t \notin A_\ell$ and
  $\tf(\widehat b) = \widehat{\tf(b)}$.
   Similarly, if $\te (b) \neq 0$, then we have $\widehat{\te(b)} = \te(\widehat b)$.

 Hence the bijection $b \mapsto \widehat b$ commutes with $\te, \tf, \teone$ and $\tfone$.
It follows that the set $M \sqcup \{0\}$ is closed under the actions $\te, \tf, \teone, \tfone$ and
$M$ is connected.
Since $\ell \in M$, we have $C(\ell) \subseteq M$ and hence $C(\ell) =M$, as desired.
\QED

\vskip1mm
\begin{example}
  In Example \ref{ex:BtensorN}(4), the element $\ell= 22122122$ is
  a $\mathfrak{q} (2)$-lowest weight vector in $\B^{\otimes 8}$. Then we
  obtain
  $A_{\ell}=\{3,4,6,7 \}$, $\widehat{\ell}=2222$ and $\overline \ell=1212$.
  We also have  $C(\ell) \cong C(2222) =C_{\gl_2}(1^4) \sqcup C_{\gl_2}(1^3 2).$
\end{example}

We close this section with a theorem that will be useful in the next
section. Let ${\mathbf a}=(a_1, \ldots, a_n)$ be a sequence of $1$'s
and $2$'s. 
We denote by $G(\mathbf{a'})$   the basis element of  $\V^{\otimes
n}$ corresponding to $[L(\mathbf{a})]$ under  $\Upsilon$, where
$\mathbf{a}' = (a_1,\ldots,a_n)' \seteq (a_1',\ldots, a'_n) $.
 We write
$\mathbf{a} x = (a_1, \ldots, a_n, x)$ for $x=1, 2$.

\vskip 3mm
Then we have the following.

\begin{theorem} \rm (\cite[Proposition 4]{BFK}, see also \cite[Theorem 3.1]{FKK}) \label{thm:upperglobalbasis}
\hfill

{\rm  Let $\bold{a}$, $\bold{a}_1$ and $\bold{a}_2$ be sequences in
$\{1,2\}$ and let   $h=v_1 \tensor v_2 -v_2 \tensor v_1$.

\begin{enumerate}
\item $G(1) = v_1  $ and $G(2) = v_2$.

\item If $\bold{a}=2 \bold{a}_1$, then $G(\bold{a})= v_2 \otimes G(\bold{a}_1)$.

\item If $\bold{a}=\bold{a}_1 1$, then $G(\bold{a})=G(\bold{a}_1) \otimes v_1$.

\item If $\bold{a}=\bold{a}_1 (12) \bold{a}_2$   with $|\bold{a}_1|=k$ and $|\bold{a}|=m$,
then $G(\bold{a})= h_{k} (G(\bold{a}_1 \bold{a}_2))$,  where
$h_{k}: \V^{\otimes m-2} \rightarrow \V^{\otimes m}$ is the
linear map given by
  $$ u_1 \otimes \cdots \otimes u_{m-2} \longmapsto
   u_1 \otimes \cdots \otimes u_k \otimes h \otimes u_{k+1} \otimes \cdots \otimes u_{m-2}.$$
 \end{enumerate} }
\end{theorem}

\begin{remark}
Let $\widetilde \Upsilon : G(\mathcal O_n) \isoto \V^{\otimes n}$ be
the identification used in \cite[\S 4.4]{BrKl}. Then we have $\psi
\circ \widetilde \Upsilon = \Upsilon$, where $\psi : \V^{\otimes n}
\rightarrow \V^{\otimes n}$ is given by $v_{a_1} \otimes \cdots
\otimes v_{a_n} \mapsto v_{a'_1} \otimes \cdots \otimes v_{a'_n}$.
Then it is not difficult to check $G(\mathbf{a'}) = \psi(\widetilde
G(\mathbf{a}))$, where $\widetilde G(\mathbf{a})$ denotes the
\emph{upper global basis (= dual canonical basis) element}
corresponding to $\mathbf{a}$, which is given in \cite{BrKl}.
\end{remark}

\vskip 5mm


\section{Categorified odd Kashiwara operators}

In this section we define the odd Kashiwara operators $\tfone,
\teone$ on ${\rm Irr} ( \mathcal O_n)$ and show that ${\rm
Irr}(\mathcal O_n)$ has a $\mathfrak{q}(2)$-crystal structure. To
define $\tfone, \teone$ we will use  tensor products with the
 infinite-dimensional irreducible highest weight
$\mathfrak{gl}_{n}$-modules $L(e_n)$  with highest weight $e_n$ and its dual
$L(e_n)^*$. The choice of
$L(e_n)$ is justified by the properties listed in the next proposition.

Recall that, for a parabolic subalgebra $\mathfrak{p}$ of
$\mathfrak{g}$, a $\mathfrak{g}$-module is {\it parabolically
induced} from a $\mathfrak{p}$-module $M_{0}$ if $M=U(\mathfrak{g})
\otimes_{U(\mathfrak{p})} M_{0}$. In this paper, we take $\mathfrak{p}$
to be the maximal parabolic subalgebra with nilradical
${\mathfrak{n}}_{\mathfrak p}$ and the   Levi subalgebra
${\mathfrak {l}}_{\mathfrak{p}} = \mathfrak{gl}_{n-1} \oplus
\mathfrak{gl}_{1}$.

 \begin{prop}   \rm \label{l-e-n} \hfill

\begin{enumerate}

\item Let $L(0) \otimes L(1)$ be the 1-dimensional
$\mathfrak{p}$-module on which $\mathfrak{n}_{\mathfrak{p}}$ acts
trivially. Then the $\mathfrak{gl}_{n}$-module $L(e_n)$ is
parabolically induced from $L(0) \otimes L(1)$. In particular,
$$\text{Supp}(L(e_n)) = \{e_n + \sum_{i=1}^{n-1} b_i (e_n - e_i)
\mid b_{i} \in \Z_{\ge 0} \}.$$

\item All the weight spaces of $L(e_n)$ are 1-dimensional.

\vskip 2mm

\item If a  $\mathfrak{gl}_{n}$-module $M$ belongs to the category
$\mathcal O$, then $M \otimes L(e_n)$ belongs to the category
$\mathcal W$.

\end{enumerate}

\end{prop}

\begin{proof}
The proofs are standard. For (1) and (2), see for
example, \cite[Lemma 11.2]{Mathieu}.
\end{proof}

Define the functors
$$
\overline{\mathcal E}_i: { \mathcal O_{i, n-i}} \to  {\mathcal W}_{i+1, n-i-1}, \;    \overline{\mathcal E}_i  := {\rm pr}_{i+1} \circ  \big(- \otimes L(e_n) \big),
$$
and set
$$
\overline{\mathcal E } : {\mathcal O_n} \to{\mathcal W_n}, \;
\overline{\mathcal E } := \bigoplus_{i=0}^n \overline{\mathcal E}_i.
$$

The following proposition plays a crucial role in defining the odd
Kashiwara operator $\teone$ on ${\rm Irr} (\mathcal O_n)$.

\begin{prop}   \rm  \label{prop-odd-kas} \hfill

\begin{enumerate}
\item  The functor $\overline{\mathcal E}$  is an exact covariant functor such that
 $$\overline{\mathcal E}: {\mathcal O_n}\longrightarrow {\mathcal
O}_{n}.$$

\item $\overline{\mathcal E} (M(a_1,...,a_n)) = \begin{cases}  M (a_1,..., a_{n-1}, 2) & \mbox{ if $a_n = 1$,} \\ 0 &  \mbox{ if $a_n = 2$.} \end{cases}$

\item $\overline{\mathcal E} (L(a_1,...,a_n)) = \begin{cases}  L (a_1,...,a_{n-1} , 2) & \mbox{ if $a_n = 1$,} \\ 0 &  \mbox{ if $a_n = 2$.} \end{cases}$
\end{enumerate}
\end{prop}
\begin{proof}

 The fact that $\overline{\mathcal E}$ is  exact and covariant is
standard. We next show that the image of $\overline{\mathcal E}$ is in $\mathcal O_n$ and prove (2).

  We would like  to show that if $M$ is in ${\mathcal
O_n}$ then $\overline{\mathcal E} (M)$ is in  ${\mathcal O_n}$ as
well. It is enough to prove  that for the projective cover $P$ of $M$,
$\overline{\mathcal E} (P)$ is in ${\mathcal O_n}$. It is clear that $\overline{\mathcal E} (P)$ is
locally $U({\mathfrak n}_+)$-nilpotent, so it remains to show that
$\overline{\mathcal E} (P)$ is finitely generated. Since every
projective in ${\mathcal O}$ has a Verma flag, we may assume that $P
= M(\lambda)$ is a Verma module. But then by Proposition
\ref{l-e-n}, $M(\lambda) \otimes L(e_n)$ has \ an infinite   filtration with
subquotients $M(\lambda + e_n + \sum_{i=1}^{n-1} b_i (e_n - e_i))$,
$b_i \in {\mathbb Z}_{\geq 0}$.  The proof of this fact uses the same reasoning as the proof of the decomposition of $M(\lambda) \otimes L(e_1)$ (for the latter, see for example \cite[Theorem 3.6]{H}). It is straightforward to check that
if $\lambda + \rho = \sum_{i=1}^n a_i e_i$ for $a_i \in \{1,2 \}$,
then  the $e_n$-coordinate of $\lambda + e_n + \sum_{i=1}^{n-1} b_i
(e_n - e_i) + \rho$ is $1$ or $2$ only if $b_1=...=b_{n-1} =0$ and
$a_n=1$. We thus proved a stronger statement: $\overline{\mathcal E}
(M(\lambda)) = M(\lambda + e_n)$ if $a_n=1$ and  $\overline{\mathcal
E} (M(\lambda)) = 0$ otherwise which implies (2).
\vskip 3mm

(3) We will use the
notation introduced at the end of Section \ref{sec:q2crystal}.
 For a sequence $\bold a=a_1 \cdots a_n$ in $\{1,2\}$, set $v_{\bold a} := v_{a_1} \otimes \cdots v_{a_n}$.
Recall that  the element $G(\bold a')$ corresponds to $[L(\bold a)]$  under $\Upsilon$.
For the case $a_n=2$, recall that $G(\bold a' 1) = \sum_{\bold b} c_{\bold b}^{\bold a'} \ v_{\bold b} \otimes v_1$ for some $c_{\bold b}^{\bold a'} \in \Z$ by Theorem \ref{thm:upperglobalbasis}(3).
Hence we have $[L(\bold a 2)] = \sum_{\bold b} c_{\bold b}^{\bold a'}  \ [M(\bold b' 2)]$.
We obtain
$\overline{\mathcal E} (L(\bold a 2)) =0$ by (2).

   In order to prove  the case $a_n=1$, it is sufficient to
prove the following statement:
\begin{equation} \label{eq:Ga2}
\parbox{30em}{if $G(\bold a 2)=\sum_{\bold b}
c_{\bold b}^{\bold a} \ v_{\bold b} \otimes  v_2 + \sum_{\bold b} d_{\bold b}^{\bold
a} \ v_{\bold b}\otimes v_1$, then $G(\bold a 1)=\sum_{\bold b} c_{\bold b}^{\bold
a} \ v_{\bold b} \otimes v_1$.}
\end{equation} 
Indeed, passing through $\Upsilon$, it implies that  if $[L(\bold a' 1)] = \sum_{\bold b}
c_{\bold b}^{\bold a} \ [M( \bold b' 1)] + \sum_{\bold b} d_{\bold b}^{\bold
a} \ [M(\bold b' 2)]$ for some $c_{\bold b}^{\bold a}, d_{\bold b}^{\bold a} \in \Z$, then
$[L(\bold a' 2)] = \sum_{\bold b} c_{\bold b}^{\bold
a} \  [M(\bold b' 2)]$.
Hence, by (2) we have
\begin{eqnarray*}
[L(\bold a' 2)] = &&\sum_{\bold b} c^{\bold a}_{\bold b} \ [M(\bold b' 2)]
=\sum_{\bold b} c^{\bold a}_{\bold b} \ [\overline{\mathcal E} (M(\bold b' 1))] \\
=&& [ \overline{\mathcal E} ] \Big ( \sum_{\bold b} c^{\bold a}_{\bold b} \ [ (M(\bold b' 1))] + \sum_{\bold b} d^{\bold a}_{\bold b} \ [M(\bold b' 2)] \Big) \\
=&&[ \overline{\mathcal E}  (L(\bold a' 1))],
\end{eqnarray*}
Thus $L(\bold a' 2)$ is isomorphic to $\overline{\mathcal E}  (L(\bold a' 1))$, as desired.

We will use induction on the length of $\bold a$. If the length of $\bold a $ is zero or 1, then it is clear
from Theorem \ref{thm:upperglobalbasis}.

  First, we consider the case $\bold a=2 \bold{a_1}$ for some $\bold{a_1}$.
   By Theorem \ref{thm:upperglobalbasis}(2),
  we have
  $G(\bold a2)=G(2\bold{a_1}2)=v_2 \otimes G(\bold{a_1}2)$ and $G(\bold a 1)=G(2 \bold{a_2}1)=v_2 \otimes G(\bold a_2 1)$.
Then \eqref{eq:Ga2} follows from the induction hypothesis.

Second, if $\bold a =1^n$, then $G(1^n2)=v_1^{\otimes {n-1}} \otimes v_1 \otimes v_2 \, - \,  v_1^{\otimes {n-1}} \otimes v_2 \otimes v_1$ and
   $G(1^{n+1})=v_1^{\otimes {n+1}}$ by Theorem \ref{thm:upperglobalbasis}(3),(4). Thus we obtain \eqref{eq:Ga2}.

Last,   let $\bold a=1^k 12 \bold{a_1}$ for some $k \ge 0$ and $\bold{a_1}$.
   By the induction hypothesis, we know that
  if $G(1^k \bold{a_1} 2)=\sum_{\bold b} c^{\bold{a_1}}_{\bold b} \ v_{\bold b} \otimes v_2 +\sum_{\bold b} d^{\bold{a_1}}_{\bold b} \ v_{\bold b} \otimes v_1$ for some $c^{\bold{a_1}}_{\bold b}, d^{\bold{a_1}}_{\bold b} \in \Z$,
  then $G(1^k \bold{a_1} 1)=\sum_{\bold b} c^{\bold{a_1}}_{\bold b}  \ v_\bold b \otimes v_1$.
  Using Theorem \ref{thm:upperglobalbasis}(4), we obtain
  \begin{align*}
  G( 1^k 12 \bold{a_1} 2)= \sum_{\bold b} c^{\bold{a_1}}_{\bold b} \ v_{\bold{b_1}} \otimes h \otimes v_{\bold{b_2}} \otimes v_2
   +\sum_{\bold b} d^{\bold{a_1}}_{\bold b} \ v_{\bold{b_1}} \otimes h \otimes v_{\bold{b_2}} \otimes v_1,
  \end{align*}
  and $$G(1^k 12 \bold{a_1}1)=\sum_{\bold b} c^{\bold{a_1}}_{\bold b} \ v_{\bold{b_1}} \otimes h \otimes v_{\bold{b_2}} \otimes v_1,$$
  where $h=v_1 \otimes v_2 -v_2 \otimes v_1$ and $\bold{b_1}$ (respectively, $\bold{b_2}$) stands for the first $k$ terms (respectively, last $|\bold{b}|-k$ terms) of $\bold{b}$.
  Therefore, we obtain \eqref{eq:Ga2}.
  \end{proof}

\begin{remark}
 Note that we have
$$\mbox{pr}_{\mathcal W_n} (M \otimes L(e_n)) = \mbox{pr}_{i+1} (M \otimes L(e_n))$$
for any $M \in \mathcal O_{i, n-i}$, by considering the support decomposition.
Hence
$$\overline{\mathcal E}(M)=\mbox{pr}_{\mathcal W_n} (M \otimes L(e_n)) $$
for any $M \in \mathcal{O}_n$.
In particular, the image $\mbox{pr}_{\mathcal W_n} (M \otimes L(e_n))$ belongs to $\mathcal{O}_n$ by Proposition \ref{prop-odd-kas} (1).
\end{remark}

In view of the above proposition, it is natural to define
$$\teone ([S]) := [ \overline{\mathcal E}(S)]   \ \
\text{for} \ \ S \in {\rm Irr}(\mathcal O_n).$$

 Now we will construct a left adjoint of $\overline{\mathcal E} $, which will be denoted by
$\overline{\mathcal F}$. We will apply the technique originally
introduced by Fiebig for Kac-Moody algebras \cite{Fieb} and later
adopted by K\r{a}hrstr\"om \cite{Kah}, to a case similar to ours .

\vskip 3mm

For $\lambda\in\frak h^*$ and a $\mathfrak{gl}_{n}$- module $M$ in
$\mathcal W$, denote by $M^{\nleqslant\lambda}$ the submodule of $M$
generated by all the weight spaces $M^\mu$ with
$\mu\not\leqslant\lambda$. Set
$$
    M^{\leqslant\lambda} := M/M^{\nleqslant \lambda}.
$$

For $i = 0,...,n$, define
$$
\overline{\mathcal F}_i : {\mathcal O}_{i; n} \to {\mathcal W_{i-1,n-i+1}}, \; \overline{\mathcal F}_i
:= {\rm pr}_{i-1} \circ  \big(- \otimes
L(e_n)^* \big)^{\leqslant(\omega_i-\rho)}
$$
(recall that  $\omega_i :=
2\sum_{j=1}^i e_j + \sum_{j=i+1}^n e_j$). Now define
$$
\overline{\mathcal F} : {\mathcal O_n} \to{\mathcal W_n}, \; \overline{\mathcal F} := \bigoplus_{i=0}^{n}
\overline{\mathcal F_{i}}.$$

\begin{prop}  \rm  \label{prop-kah}Let $\lambda \in {\mathfrak h}^*$.
\begin{enumerate}
\item  The functor $M\mapsto M^{\leqslant\lambda}$
    is right exact on $\mathcal W$.

\item  If $M$ belongs to  $\mathcal O$,
then $\left( M \otimes L(e_n)^* \right)^{\leqslant\lambda}$ belongs to $\mathcal O$ as well.

\item  The functor $\overline{\mathcal F}_i$
is the left adjoint of the functor $\overline{\mathcal E}_{i-1}$. Furthermore, we have
$$\overline{\mathcal F}_i(M(a_1, \ldots, a_n)) = \begin{cases}  M(a_1, \ldots, a_{n-1}, 1) \ \ & \text{if} \ a_{n}=2, \\
0 \ \ & \text{if} \ a_{n}=1.  \end{cases}
$$
\item The  functor $\overline{\mathcal F}$ is the left adjoint of $\overline{\mathcal E}$.
\end{enumerate}
\end{prop}

\begin{proof}
Part (1) is \cite[Lemma 2.9]{Kah}, while part  (2) is \cite[Corollary 2.12]{Kah}.
 For part  (3) we follow the proof of \cite[Theorem 3.4]{Kah}.
  Note that Theorem 3.4 in \cite{Kah} is for the principal block ${\mathcal O}_0$ of ${\mathcal O}$, namely for the functor $M \mapsto {\rm pr}_{{\mathcal O}_0} \left( M \otimes L(e_n)^* \right)^{\leqslant 0}$ but the same reasoning applies for the block ${\mathcal O}_{i;n}$.
  To find  ${\rm pr}_{i-1} \left( M(a_1,...,a_n) \otimes L(e_n)^* \right)^{\leqslant(\omega_i-\rho)}$ we first use Proposition \ref{l-e-n} and fix a basis $v_{b}$, $b = (b_1,...,b_{n-1}) \in \left({\mathbb Z}_{\geq 0}\right)^{n-1}$, such that $\wt (v_b) = e_n + \sum_{j=1}^{n-1} b_j (e_n - e_j)$. Then the set $\{ v_b^* \; | \; b \in \left({\mathbb Z}_{\geq 0}\right)^{n-1} \}$ forms a basis of $L(e_n)^*$. Thus, if $v$ is a highest weight vector of $M(a_1,...,a_n)$ then
$$
M(a_1,...,a_n) \otimes L(e_n)^* = \bigoplus_{b} U({\mathfrak n}_-) (v \otimes v_b^*)
$$
as $U({\mathfrak n}_-)$-modules. Now using \cite[Proposition 2.10]{Kah} we have that
$$
\left( M(a_1,...,a_n) \otimes L(e_n)^*\right)^{\leqslant(\omega_i-\rho)}= \bigoplus_{{\rm wt}(v \otimes v_b^*) \leqslant \omega_i-\rho} U({\mathfrak n}_-) (v \otimes v_b^*)
$$
Since the $e_n$-coordinate of $\wt(v \otimes v_b^*) + \rho$ is $a_n - 1 - \sum_{j=1}^{n-1}b_j$, we have that $\wt(v \otimes v_b^*) + \rho \leq \omega_i$ only if $a_n - 1 - \sum_{j=1}^{n-1}b_j \geq 1$. Hence $a_n = 2$ and $b_1=...=b_{n-1} = 0$. This completes the proof of  (3). Part (4) follows from part (3).
\end{proof}

\vskip 3mm
 Set
$$\tfone([S]) :=[\rm{hd}  \overline{\mathcal F} (S)].$$
\begin{remark}
One easily checks that even for $n=2$, $ [{\rm hd}\, \overline{\mathcal F}(S)]$ might be different from $ [\overline{\mathcal F}(S)]$. Indeed, if $S = L(2,2)$, then by Proposition \ref{prop-kah}(3),
$$
[\overline{\mathcal F}(L(2,2))] = [\overline{\mathcal F}(M(2,2))] = [M(2,1)] = [L(2,1)] + [L(1,2)].
$$
\end{remark}

\begin{lemma} \label{lem-odd-kas-f}  \rm{
For $a_1,\ldots ,a_{n-1} \in \{1,2\}$, we have
\begin{align*}
&\overline{\mathcal F}(L(a_1,\ldots,a_{n-1},1))=0 \ \text{and} \
{\rm hd}\, \overline{\mathcal F}(L(a_1,\ldots,a_{n-1},2)) = L(a_1,\ldots,a_{n-1},1).
\end{align*}}
\end{lemma}
\begin{proof}
By Proposition \ref{prop-kah}(3), we know that $\overline{\mathcal F}$ maps a simple module in ${\mathcal O}_n$ to a highest weight module in ${\mathcal O}_n$ or $0$. Hence $\overline{\mathcal F}(S)$ has a simple head for $S \in {\rm Irr}( {\mathcal  O}_n )$,  if it is nonzero.  Now the assertion follows from  Proposition \ref{prop-odd-kas}(3) and Proposition \ref{prop-kah}(4). \end{proof}

\begin{theorem} \label{main} \hfill
\rm
\begin{enumerate}

\item There is a $\mathfrak{q}(2)$-crystal structure on
$\text{Irr}(\mathcal O_n)$ with odd Kashiwara operators $\teone$ and
$\tfone$ given above.

\item As a $\mathfrak{q}(2)$-crystal, $\text{Irr}(\mathcal O_n)$ is
isomorphic to $\B^{\otimes n}$.

\end{enumerate}
\end{theorem}

\begin{proof}
Let $\psi$ be the map given by $(a_1, \ldots, a_n) \mapsto a_1'
\otimes \cdots \otimes a_n'$, where $a_i=1$ or $2$, $1'=2$,
$2'=1$.

  For part (1), we use Proposition \ref{prop-odd-kas}(3) and Lemma \ref{lem-odd-kas-f}.

For part (2),
one can easily check that $x [L(a_1,...,a_n)] = [L(\psi^{-1}  x \psi
(a_1,...,a_n))]$ for $x = \tfone, \teone$, whenever $x [L(a_1,...,a_n)]
\neq 0$.
\end{proof}

\vskip 5mm


\section{Invariants of  connected components}
One of the important properties of the $\mathfrak{gl}_2$-crystal
structure of ${\rm Irr} (\mathcal O_n)$ is that the isomorphism
classes of simple objects in a fixed  parabolic subcategory of
$\mathcal O_n$ form a $\mathfrak{gl}_2$-subcrystal of ${\rm Irr}
(\mathcal O_n)$. A similar but slightly weaker statement holds for
the $\mathfrak{q} (2)$-crystal ${\rm Irr} (\mathcal O_n)$. To
formulate this statement, we need to introduce some notation.

For a sequence ${\mathbf a}=(a_1,...,a_n)$ of $1$'s and $2$'s, let
$$
I_{\rm fin} (a_1,...,a_n) := \{ i  \; | \; a_i = 2 \mbox{ and
}a_{i+1} = 1\}.
$$
In particular, $I_{\rm fin} (a_1,...,a_n)$ is a subset of
$\{1,...,n-1 \}$. Recall that for an irreducible
$\mathfrak{gl}_{n}$-module $M$ in $\mathcal W$ and a root $\alpha$ of
$\mathfrak{gl}_{n}$, every root vector $x$ in the $\alpha$-root space
acts either  injectively  or locally finitely on $M$.  Indeed, this follows from the fact that the set of all $m$ in $M$ for which $x^Nm = 0$ for sufficiently large $N\geq 1$ forms a submodule of $M$.
 For a module $L$ in the category $\mathcal O$, we define
$\Pi_{\rm fin}(L)$ to be the set of simple roots $\alpha$ such that the
 vectors in the $(-\alpha)$-root space act locally finitely on $L$.

\vskip 2mm

For a subset $I$ of $\{1,...,n-1 \}$, denote by ${\mathcal O}_I$ the
parabolic subcategory of ${\mathcal O}$ consisting of all
$\mathfrak{gl}_{n}$-modules $M$ on which the root vectors of $-e_i +
e_{i+1}$  $(i \in I)$ act locally finitely.
Some properties of ${\mathcal O}_I$ related to the $\mathfrak{gl}_{2}$-crystal structure of $\text{Irr}(\mathcal O_n)$  are listed in the following
proposition. We refer the reader to
\cite[Chapter 9]{H} for other important properties of ${\mathcal O}_I$.

\vskip 3mm

\begin{prop}  \rm \label{prop-par-gl} Let $a_i=1$ or $2$   for $i = 1,...,n$.

\begin{enumerate}
\item  $\Pi_{\rm fin} (L(a_1,...,a_n)) = \{ e_i - e_{i+1} \; | \; i \in I_{\rm fin} (a_1,...,a_n)\}.$

\item Let $L$ be an irreducible  $\mathfrak {gl}_{n}$-module whose isomorphism
class belongs to the connected component $C([L(a_1, \ldots, a_n)])$ in
the $\mathfrak{gl}_{2}$-crystal $\text{Irr}(\mathcal O_n)$.

Then  $\Pi_{\rm fin} (L) =\Pi_{\rm fin} (L(a_1,...,a_n))$. In
particular, $L$ belongs to ${\mathcal O}_{I}$, where $I = I_{\rm
fin} (a_1,...,a_n)$.

\item  For every subset $I$ of $\{1,...,n-1 \}$, the isomorphism classes of
irreducible $\mathfrak{gl}_{n}$-modules in $\mathcal O_n \cap
{\mathcal O}_I$ form a $\mathfrak{gl}_2$-subcrystal of ${\rm Irr}
(\mathcal O_n)$.

\end{enumerate}
\end{prop}
\begin{proof}
Part (1) is a standard fact.  For parts (2) and (3), we  use Theorem \ref{thm:gl2-crystal}  or the fact that if $\alpha \in \Pi_{\rm fin} (L)$ and
 $x$ is in the $(-\alpha)$-root space then $x$ acts locally finitely on $L \otimes L(e_1)$ and $L \otimes L(e_1)^*$.
 \end{proof}

\vskip 2mm

The $\mathfrak{q} (2)$-version of the above proposition is the
following.

\vskip 2mm

\begin{prop}  \rm
Let $a_i=1$ or $2$ for $i = 1,...,n$.

\begin{enumerate}

\item If $\overline{\mathcal E}(L(a_1, \ldots, a_n)) \neq 0$ (equivalently, $a_n=1$), then
$$\Pi_{\rm fin}(\overline{\mathcal E}(L(a_1, \ldots, a_n))) =
\Pi_{\rm fin}(L(a_1, \ldots, a_n)) \setminus \{e_{n-1} - e_{n} \}.$$

\item
Let $L(b_1, \ldots, b_n)$ $(b_i=1, 2)$ be the irreducible
$\mathfrak{gl}_n$-module whose isomorphism class belongs to the
connected component $C([L(a_1, \ldots, a_n)])$ in the
$\mathfrak{q}(2)$-crystal $\text{Irr}(\mathcal O_n)$.

\vskip 2mm

Then $L(b_1, \ldots, b_n)$ belongs to ${\mathcal O}_{I}$, where
$I=I_{\rm fin} (a_1,...,a_n) \setminus \{ n-1\}$.

\vskip 3mm

\item For every subset $I$ of $\{1,...,n-2 \}$, the
isomorphism classes of irreducible $\mathfrak{gl}_{n}$- modules in
$\mathcal O_n \cap {\mathcal O}_I$ form a  ${\mathfrak q}
(2)$-subcrystal of ${\rm Irr} (\mathcal O_n)$.

\end{enumerate}
\end{prop}

\begin{proof}
Part (1) follows from Theorem \ref{main}(2) and Proposition
\ref{prop-par-gl}(1). Parts  (2) and (3) follow from (1).
\end{proof}

 We finish this section with a result on the decomposition of the ${\mathfrak q} (2)$-connected components of $\B^{\otimes n}$ into $\mathfrak{gl}_2$-connected components.

\begin{prop} \rm Let $\ell$ be a ${\mathfrak q}(2)$-lowest weight vector in $\B^{\otimes n}$ with $|\widehat \ell| \ge
2$. Then $C([L(\ell')]) = A \sqcup B$, where $A$ and $B$ are the
following    $\mathfrak{gl}_2$-subcrystals, which are connected in
$\rm{Irr}(\mathcal{O}_n)$.
\begin{eqnarray*}
 A & = & \{ [L(a')] \; |\; \overline a = \overline \ell,  I_{\rm fin}(\widehat a')  = \emptyset   \}, \\
 B& = & \{ [L(a')] \; | \; \overline a = \overline \ell, I_{\rm fin}(\widehat a') = \{ | \widehat \ell | - 1\}  \}.
\end{eqnarray*}
\end{prop}
\begin{proof}
 By Theorem \ref{main}(2), we can use the description of $\B^{\otimes n}$ in Section \ref{sec:q2crystal}.
Then by Proposition \ref{prop-c-l}(2),  we may assume that $\ell =2^n$ and hence we obtain $C([L(\ell')]) = C_{\mathfrak{gl}_2}([L(2^n)]) \sqcup C_{\mathfrak{gl}_2}([L(2^{n-1} 1)])$ by Proposition \ref{prop-easy-dec}.
The statement follows  from  Proposition 5.1(2).
\end{proof}

\end{document}